\documentclass[11pt,a4paper,reqno]{amsart}
\usepackage{amsmath}
\usepackage{amsfonts}
\usepackage{graphicx}
\usepackage{graphics}
\usepackage{epsfig}
\usepackage{anysize}
\usepackage{lipsum}
\usepackage{wrapfig}
\usepackage{pgf,tikz}
\marginsize{3cm}{3cm}{3cm}{3cm}
\newtheorem{theorem}{Theorem}[section]

\newtheorem{definition}{Definition}[section]

\newtheorem{note}{Note}[section]
\newtheorem{notation}{Notation}[section]
\makeatletter

\begin{document}
	
	\begin{center}
		\Large{\bf  Herscovici Conjecture on Pebbling }
	\end{center}    
	\begin{center}
		 I. Dhivviyanandam$^{1}$, A. Lourdusamy$^{2}$  S. Kither Iammal$^{3}$, K.Christy Rani$^4$\\
		  $^1$ Department of Mathematics, North Bengal St.Xavier's College, Rajganj, West Bengal, India.\\
		  e-mail: divyanasj@gmail.com; https://orcid.org/0000-0002-3805-6638.\\
		 
		  $^2$  Department of Mathematics, St.Xavier's 
		 College (Autonomous), Palayamkottai-627002, Tamilnadu, India.\\
		  e-mail: lourdusamy15@gmail.com; https://orcid.org/0000-0001-5961-358X.\\
		 $^3$ Department of Mathematics, Jayaraj Annapackiam College for Women (Autonomous), Periyakulam, Tamilnadu,  India.\\
		  e-mail: cathsat86@gmail.com; https://orcid.org/0000-0003-3553-0848.\\
		  $^4$ Department of Mathematics, Bon Secours College for Women (Autonomous), Thanjavur-613006 Tamilnadu,  India.\\
		 e-mail: christy.agnes@gmail.com; https://orcid.org/0000-0002-2837-9388.\\
	
	\end{center}        
	\begin{abstract}
		
			Consider a configuration of pebbles on the vertices of a connected graph. A pebbling move is to remove two pebbles from a vertex and to place one pebble at the neighbouring vertex of the vertex from which the pebbles are removed.
		
		For a positive integer $t$,	with every configuration of $\pi_t(G)$(least positive integer) pebbles, if we can transfer $t$ pebbles to  any target through a number of pebbling moves then $\pi_t(G)$ is called the $t$-pebbling number of $G$. 
		 
		We discuss the computation of the $t$-pebbling number, the $2t-$ pebbling property  and Herscovici conjecture  considering  total graphs.\\
		
		
		
	\end{abstract}
\section{Introduction}
Pebbling was introduced by Lagarias and Saks. F.R.K. Chung\cite{1}  developed the fundamental concepts on pebbling and published the results in 1989. Hulbert published a survey article on pebbling that increased the interest of researchers to further explore graph pebbling\cite{2}.

Let us denote $G'$s vertex and edge sets as $V (G) $and $E (G) $, respectively. 	Consider a configuration of pebbles on $V(G)$ of a connected graph $G$. A pebbling move is to remove  2 pebbles from a vertex and to place 1 pebble at the neighbouring vertex of the vertex from which the pebbles are removed. The pebble number of $v$ in $G$ is the minimum positive number $\pi(G, v)$ that allows us to shift a pebble to $v$ using a sequence of pebble transformations, regardless of where these pebbles are located on the vertices of G. The pebbling number, $\pi(G)$, of a graph $G$ is the maximum $\pi(G, v)$ over all the vertices of a graph.

For a positive integer $t$,	with every configuration of $\pi_t(G)$(least positive integer) pebbles, if we can transfer $t$ pebbles to  a target through a number of pebbling moves then $\pi_t(G)$ is called the $t-$pebbling number of $G$. Note that $\pi_1(G)= \pi(G)$.

For the configuration  on $G$, let $q$ denote the count of vertices with at least one pebble.  A graph $G$  has  the 2t-pebbling property if $2t$ pebbles can be transfered to a destination  with a minimum of $2\pi_t(G)- q+1$ pebbles.

Lourdusamy et al\cite{3}, \cite{4} studied  that  $n$-cube graphs, the complete graphs, the even cycle graphs,  the complete $r$-partite graphs, the star graphs, the wheel graphs, and the fan graphs have the $2t$-pebbling property.  We discuss the $t-$pebbling number, the $2t-$pebbling property  and Herscovici conjecture  considering  total graphs and some families of chemical graphs.

\section{Preliminaries}
For the basic definitions in graph theory, the reader can read \cite{9}.

\begin{definition}\cite{10}
	The total graph is obtained through a graph operation relating the line graph. A total graph, $T(G)$, of a graph $G$ is a graph with following properties. \\
	(i) whose vertices can be put in one-to-one correspondence with the elements of the set $V(G)\cup E(G)$.\\
	(ii)  Two vertices of $T(G)$ are adjacent if the corresponding elements in $G$ are two adjacent vertices, two adjacent edges or an incident vertex and edge.
	
\end{definition}
\begin{definition}\cite{10}
	Let $V(G\times H) = V(G)\times V(H)$. The edge set of $G\times H$ is defined as: $E(G\times H)= \{((x,y), (x^{'}, y^{'})): x=x^{'}\ and \ (y,y^{'})\in E(H)\ or \ (x,x^{'})\in E(G)\ and \ y=y^{'}\} \ \ and \ (y,y^{'})\in E(H)\ or \ (x,x^{'})\in E(G)\ and \ y=y^{'}\}$. The graph $G\times H$ is called the Cartesian product of $G$ and $H$.	Let $p_{k}$ denote the count of pebbles on ${x_k}\times H,\ q_{k}$  denote the count of occupied vertices in $\{x_k\}\times H$.
\end{definition}
\begin{note}
	For the transmitting subgraph  the reader can refer to \cite{8}.
\end{note}

\begin{notation}
Let $p\left(v\right)$  denote the count of pebbles on $v$ and $p^{\sim}(v)$  the count of pebbles on $v$  which are not on the defined path $P$. Let $S\subset V$. Let $ p^{\sim}(S)$ be the the count  of pebbles  on the vertices not in $S$.  Further, we denote $P_{k}$ as  $k^{th}$ path and $P_{k}^{\sim}$ be the vertices which are not on $k^{th}$ path $P$, where $k$ is the positive integer.  $(x_{k}) \underrightarrow{t} (x_{l})$ refers to  taking off at least $2t$ pebbles from $(x_{k})$ and transferring at least $t$ pebbles on $(x_{l})$. We use $\beta$ to denote the destination vertex.
\end{notation}

The aim of this paper is to verify  Herscovici's conjecture,  when   $G$ is  $T(P_n)$\ and  $H$ satisfies the $2t-$pebbling property,  $\pi_{st}\left(G\times H\right)\le \pi_s\left(G\right)\pi_t\left(H\right)$.  We prove the Herscovisi conjecture  for the product of  the total graph of the path. Herscovici's conjecture is evolved from Graham's Conjecture and Lourdusamy's Conjecture. The first  conjecture is  $\pi\left(G\times H\right)\le \pi	\left(G\right)f\left(H\right)$. The second conjecture is  $\pi_t\left(G\times H\right)\le f\left(G\right)\pi_t\left(H\right)$. From these two conjectures we get Herscovici conjecture:  $\pi_{st}\left(G\times H\right)\le \pi_s\left(G\right)\pi_t\left(H\right)$.
\begin{center}
\section{Main Results}
\end{center}

\section{Pebbling Number of the Total Graph of Path}
\begin{theorem}\label{total}
For $(T(P_n))$ , $\pi(T(P_n))=2^{n-1}+(n-2)$.
\end{theorem}

\begin{proof}
Let the vertices of $T(P_n)$ be $\{x_1,\ x_2,\ x_3,\cdots, \ x_n, y_{12},y_{23},\cdots , y_{(n-2)(n-1)},\ y_{(n-1)(n)}\}$ and the edges be $\{x_ix_{i+1},\ y_{j(j+1)}y_{(j+1)(j+2)}\ x_iy_{(i)(i+1)},\ x_{i+1}y_{i(i+1)}\}$ where $1\leq i \leq n-1$\ and $1\leq j\leq n-2$. Let $A= \{x_1,\ x_2,\ x_3,\cdots, \ x_n\}$ and $B= \{y_{12},y_{23},\cdots , y_{(n-2)(n-1)},\ y_{(n-1)(n)}\}$. Let $\beta= x_1$.  Let us place one pebble each on the vertices in $<V(B)-\{y_{12}, y_{(n-1)(n)}\}>$ and  $2^{n-1}-1$ pebbles on $x_n$, it is not possible to transfer a pebble to $\beta$. Hence, $\pi(T(P_n))\geq 2^{n-1}+(n-2)$.\\

Now we prove $\pi(T(P_n))\leq 2^{n-1}+(n-2)$. Let $D$ be a configuration of $2^{n-1}+(n-2)$ pebbles on  $T(P_n)$.\\

\noindent \textbf{Case 1:} Let $\beta= x_k \ or \ y_{s(s+1)}$, where $k={1,n} \ and \ j={1,n-1}$. \\
Let $\beta = x_n$. Stacking $2^{n-1}$ pebbles on $x_1$ we are done. If we place one pebble each on the vertices in $B$ other than $y_{n(n-1)}$ and 4 pebbles on $x_1$ we can shift a pebble to $\beta$. If we place one pebble each on the vertices in  $\{<V(A)-\{x_1, x_n\}>\}$  and 2 pebbles on $x_1$ or  if any vertex of the set $B$ contains 2 pebbles then using the transmitting subgraph we can transfer a pebble to $\beta$. If we place 2 pebbles each on the vertices of the set $B$ other than $y_{n(n-1)}$ then we can transfer one pebble each to the vertices of the set $A$ from $x_{2}$ to $x_{n(n-1)}$. By using 2 pebbles on $x_1$ we can transfer a pebble to $\beta$. If we put 2 pebbles each on the vertices in $A$ other than $x_1,  x_{n-1},\ and \ x_n$, we can transfer one pebble each to the vertices of the set $A$ from $x_3$ to $x_{n-1}$ or to the vertices of the set $B$ from $y_{23}\ to \ y_{(n-1)(n-2)}$. Thus, by placing 4 pebbles on the vertex $x_1$  or $y_{12}$ we can reach $\beta$. For the first case, the total number of pebbles used is $2(n-2)+2$ and for the second case the total number of pebbles used is $8+2(n-3)$. Let $p(y_{n(n-1)})=1\ and \ p(x_{n-1})=1$. Then placing $2^{n-2}$ pebbles on $x_1$ or $y_{12}$ we can transfer a pebble to $\beta$. By symmetry, we can prove for $\beta=x_1,\ y_{n(n-1)}\ and \ y_{12}.$ Let $p(y_{n{n-1}})=0\ and \ p(x_{n-1})=1$. Then placing $2^{n-2}$ pebbles on $x_1$ we can shift one more additional pebble to $x_{n-1}$ and reach a pebble to $\beta$. Let $p(y_{n{n-1}})=1\ and \ p(x_{n-1})=0$. Then placing $2^{n-2}$ pebbles on $y_{12}$ we can transfer a pebble to $\beta$. Let $p(y_{n{n-1}})=0\ and \ p(x_{n-1})=0$. If $p(V(A))=2^{n-2}$ and $p(V(B))=2^{n-2}$, we can transfer 2 pebbles to the adjacent vertices of $\beta$. Thus, we can transfer a pebble to $\beta$. By symmetry, we can prove for  $\beta=x_1,\ y_{n(n-1)}\ and \ y_{12}.$

\noindent \textbf{Case 2:} Let $\beta=x_a \ or \ y_{b(b+1)}$, where $2\geq a \geq n-1 \ and \ 2\geq b \geq n-2 $. \\
Without loss of generality, let $\beta=x_a$. Placing $2^{n-a}$ pebbles on $x_n$  or $2^{a-1}$ pebbles on $x_1$ we are done, since the distance from $x_n$ to any other vertex is at most $n-a$ and from $x_1$ to any other vertices is at most $a-1$ when $x_a$ is the target. Thus, the total number of pebbles used is $2^{n-a}+2^{a-1}+(n-2) < 2^{n-1}+(n-2)$. By symmetry, we can prove for $\beta=y_{b(b+1)}$.  Hence, $\pi(T(P_n))= 2^{n-1}+(n-2)$

\end{proof}
\section{$t-$ pebbling number of total graph of path}
\begin{theorem}
For the graph $T(P_n)$, $\pi_t(T(P_n))=t(2^{n-1})+(n-2)$.
\end{theorem}
\begin{proof}

Placing $t2^n-1+(n-2)$ pebbles on $x_n$ and  one pebble each on the vertices in $V(B)-\{y_{12}, y_{(n-1)(n)}\}$ we cannot transfer  $t$ pebbles to the destination $x_1$. Hence, $\pi_t(T(P_n))\geq t2^{n-1}+(n-2)$.\\
We prove it by induction on $t$.
Now consider a configuration of $t2^{n-1}+(n-2)$ pebbles on $V((T(P_n)))$. By Theorem \ref{total} ensures the validity of the result for $t=1$. Assume it  for $2\leq t^{'}< t$.\\

\noindent\textbf{ Case 1:} Let $p(\beta)=0$\\
Clearly, $V(<(T(P_n))-{\beta}>)\ \underrightarrow{1} \beta$ by using at most $2^{n-1}$ pebbles. Let $p(V<(T(P_n))-{\beta}>)= t2^{n-1}+(n-2)- 2^{n-1}= (t-1)2^{n-1}+(n-2)$.  By induction, we can transfer $t_1$ additional pebbles to $\beta$ .\\

\noindent\textbf{ Case 2:} 	Let $p(\beta)=x$ where $1\leq x \leq t-1$.\\
Let $p(V<(T(P_n))-{\beta}>) = t2^{n-1}+(n-2)- x$. Since $t2^{n-1}+(n-2)- x \geq (t-x) 2^{n-1}+(n-2)- 2^{n-1}$, we can transfer $t-x$ additional pebbles to $\beta$.
\end{proof}

\section{The 2- Pebbling Property of the total graph of Paths }
\begin{theorem}
The graph $(T(P_n))$ satisfies the two-pebbling property.
\end{theorem}
\begin{proof}
Let $A= \{x_1,\ x_2,\ x_3,\cdots, \ x_n\}$ and $B= \{y_{12},y_{23},\cdots , y_{(n-2)(n-1)},\ y_{(n-1)(n)}\}$. 
We use induction on $n$. Clearly, the result is true for $n=2$, since the graph is isomorphic to $K_3$ and it satisfies the two-pebbling property. Assume the result for $3\leq n^{'}< n$.

Let $S$ be a distribution of $2\pi(T(P_n))-q+1=2(2^{n-1}+(n-2))-q+1$ pebbles on  $V((T(P_n)))$.\\

\noindent\textbf{ Case 1:} Let $\beta=x_1$.\\
Let $p(\beta)=0, \ p(x_2)\leq 1\ and \ p(y_{12})\leq 1$. If all the vertices receive one pebble each then $q=2n-2$. Then by placing 2 pebbles on any one of the vertices of $T(P_n)$ we can transfer a pebble to $\beta$. The total number of pebbles used to transfer a pebble to $\beta$ is $2n-1$. Further, to move a pebble we have  $2(2^{n-1}+(n-2))-q+1- (2n-1)=2(2^{n-1}+(n-2))-2n+2+1-2n+1=2(2^{n-1}+(n-2))-n+1)$ pebbles. If we place $2^{n-1}$ pebbles on any one of the vertices of $T(P_n)$, we are done. Thus, we used $2^{n-1}+2n-1$ pebbles.\\
Let $q< 2n-2$. Then $x_2$ or $y_{12}$ is occupied. Let $p(x_2)=1$ and $p(y_{12})=0$. We can transfer a pebble to $x_2$ with a maximum cost of $2^{n-2}$ pebbles and hence we are done.  Now the total number of pebbles available to move an additional pebble to $\beta$ is at least $2(2^{n-1}+(n-2))-q+1- 2^{n-2}-1=2^{n}+2n-4))-q+1-2^{n-2}-1=2^n-2^{n-2}+2n-2-q= 2^{n-2}(4-1)+2n-2-q\geq 2^{n-1}+(n-2)= \pi(T(P_n)$. By Theorem\ref{total}, we can transfer another pebble to $\beta$  using $2^{n-1}$ pebbles and could retain at least $n-1$ pebbles. Similarly, we can prove when $p(x_2)=0$ and $p(y_{12})=1$. \\

Assume $p(x_2)=0$ and $p(y_{12})=0$. If there is vertex in $N(x_2)$ with at least 2 pebbles  we can shift a pebble to $x_2$ and  an additional pebble to $x_2$ at a maximum cost of $2^{n-2}$ pebbles, since $d(x_2, u) \leq n-2$ for any $u\in V(T(P_n))$. Thus we can put 2 pebbles to $x_2$ and transfer a pebble to $\beta$. Now the total number of pebbles available to transfer a pebble to $\beta$ is at least $2(2^{n-1}+(n-2))-q+1- 2^{n-2}-1=2^{n}+2n-4))-q+1-2^{n-2}-1=2^n-2^{n-2}+2n-2-q= 2^{n-2}(4-1)+2n-2-q\geq 2^{n-1}+(n-2)= \pi(T(P_n)$. Since $q\leq 2n-3$ by Theorem\ref{total}, we can shift a pebble to $\beta$ at the maximum cost of $2^{n-1}$ pebbles and we could retain at least $(n-2)$ pebbles.

Let $p(N(y_{12}))\leq 1$.  Consider $p(N(y_{12}))= 0$. In this case, we distribute  $2(2^{n-1}+(n-2))-q+1$ pebbles on  $<V(T(P_n))-\{x_1,x_2,y_{12},y_{23}\}>$. Let $p(y_{23})=0$. Using $2(2^{n-2}+(n-2))-q+1$ pebbles we can shift 2 pebbles to $x_2$ and we could retain at least $(n-1)$ pebbles. Therefore, we are done. Now the total count of pebbles available to transfer another pebble to $\beta$ is at least $2(2^{n-1}+(n-2))-q+1- (2(2^{n-2}+(n-2))-q+1+(n-1)=(2^{n-1}+(n-2)+1=\pi(T(P_n))+1$. By Theorem \ref{total}, it can be done using at most $2^{n-1}$ pebbles and we could retain at least $(n-1)$ pebbles.

Consider $p(N(y_{23}))= 1$. Since $p(y_{12})= 0$  using  $2(2^{n-1}+(n-2))-(q-1)+1$ pebbles  we can shift 2 pebbles to $x_2$ and keep  at least $(n-1)$ pebbles in the graph. Therefore, one pebble can be shifted to $\beta$. Now the total count of  pebbles available to transfer an additional pebble to $\beta$ is at least $2(2^{n-1}+(n-2))-q+1- (2(2^{n-2}+(n-2))-(q-1)+(n-1)=(2^{n-1}+(n-2)=\pi(T(P_n))$. By Theorem  \ref{total}, we can shift an additional pebble to $\beta$ using at most $2^{n-1}$ pebbles. And we could retain at least $(n-2)$ pebbles. Similarly, we can prove for $\beta= x_n, y_{n-1}, y_{12}$. \\

\noindent\textbf{ Case 2:} Let $\beta=x_k\ or\ y_{(n-1)(n-2)} $, where $k={2,n-1}$.\\
let $\beta=x_k$  and $p(x_2)=0,\ p(x_1)\leq 1, p(x_3)\leq 1, p(y_{12})\leq 1 \ and \ p(y_{23})\leq 1$. Now the number of pebbles distributed on $<T(P_n)-\{x_1,y_{12}\}>$ is at least $2(2^{n-1}+(n-2))-q+1-2\geq 2(2^{n-2}+(n-2))-q_1+1$ where $q_1$ is the count of the occupied vertices in $<T(P_n)-\{x_1,y_{12}\}>$. Clearly $q>q_1$. Then we can transfer 2 pebbles to $x_2$ using induction. By symmetry, we are done if $\beta= x_{n-1}, \ and\ y_{(n-1)(n-2)}$.\\

\noindent\textbf{ Case 3:} Let $\beta= x_s\ or\ y_{k(k+1)}$, where $s=\{3,4,\cdots,n-2\}\ and \ k=\{3,\cdots n-3\}$.\\
Let $\beta= x_s$. Suppose $p(x_1)+p(y_{12})\geq 2^{n-2}+(n-2)$ pebbles, we note that $<V(T(P_n))-\{x_{n-1}, x_n, y_{n(n-1)}, \\ y_{(n-1)(n-2)}\}>$ is isomorphic  to $T(P_{n-2})$. Since the count of  pebbles on  $<V(T(P_n))-\{x_{n-1}, x_n, y{n(n-1)}, y){(n-1)(n-2)}\}>$ is at least $2^{n-2}+(n-2)= 2(2{n-3})+(n-2)$ we can shift 2 pebbles to $\beta$ by Theorem \ref{total}.

Suppose $p(x_1)+p(y_{12})< 2^{n-2}+(n-3)$  then the subgraph induced by 
$<T(P_n)-\{x_1,y_{12}\}>$ has at least $2(2^{n-1}+(n-2))-q+1-2^{n-2}+(n-3)\geq 2(2^{n-2}+(n-1))-q_1+1$ pebbles where $q_1$ is the number of occupied vertices in $<T(P_n)-\{x_1,y_{12}\}>$. Since $q>q_1$, we can move 2 pebbles to$\beta$ by induction. \\
\end{proof} 
\section{The 2t- Pebbling Property of Total graph of Path graphs }
\begin{theorem}
The graph $T(P_n)$ has the $2t$-pebbling property.
\end{theorem}
\begin{proof}
By Theorem\ref{total}  validates the  result for $t=1$. Assume it for  $2\leq t^{'}< t$. Let $p(V(T(P_n))=2\pi_t(T(P_n))-q+1= 2(t2^{n-1}+(n-2))-q+1$. Let $\beta$ be the destination vertex. \\

\noindent \textbf{Case 1:} Let $p(\beta)=x$ where $1\leq \ x \leq \ 2t-1$.\\
Let $p(<T(P_n)-{\beta}>)$ = $2(t2^{n-1}+(n-2))-q+1-x\geq (2t-x)2^{n-1}+n-2=\pi_{2t-x}(T(P_n))$. Therefore, we can shift $2t-x$ pebbles additionally to $\beta$ by Theorem \ref{total}.\\
\noindent \textbf{Case 2:} Let $p(\beta)=0$.\\
It is easy to transfer 2 pebbles to any target $\beta$ in $T(P_n)$ with the maximum cost $2^{n}$ pebbles. The number of pebbles remaining on  $V(T(P_n))$ is at least $2(t2^{n-1}+(n-2))-q+1-2^{n}=2(t-1)2^{n-1}-q+1$ Thus, we can transfer $2(t-1)$ pebbles to  $\beta$ by induction.

\end{proof}
\section{Herscovici Pebbling Conjecture on the product of total graph of Path graphs}

\begin{theorem}
For $n, m\geq 2,$ we have $\pi_{st}(T(P_n))\times (T(P_m))\leq \pi_s(T(P_n))\pi_t(T(P_m))= (s(2^{n-1}+(n-2))) (t2^(m-1)+(m-2))$.
\end{theorem}

\begin{proof} Let $G=(T(P_n))$ and $H= T(P_m))$. 
Let $V(G) = \{x_{1}, x_{2}, x_{3},\cdots ,x_{n},\\ y_{1}, y_{2},y_{3} \cdots y_{n-1}\}$ and 
let $V(H)=\{v_{1}, v_{2}, v_{3},\cdots ,v_{m}, u_{1}, u_{2},u_{3}\\
\cdots u_{m-1} \}$.  Let $S_1=\{x_{1}, x_{2}, x_{3},\cdots ,x_{n}\}; \ S_2=\{y_{1}, y_{2},y_{3}, \cdots, y_{n-1}\};\ T_1=\{v_{1}, v_{2}, v_{3},\cdots ,v_{m}\}; T_2=\{u_{1}, u_{2},u_{3}, \cdots ,u_{m-1} \}$.  
Let  $p_{i}^{1}$ stand for the total count  of pebbles on $\{V(S_1)\times V(H)\},\ q_{i}^{1}$ for the total count  of occupied vertices of $\{V(S_1)\times V(H)\}$ where $1\leq i\leq n$. Let  $p_{j}^{2}$ stand for the total count of  pebbles on $\{V(S_2)\times V(H)\},\ q_{j}^{2}$ for the total count  of occupied vertices of $\{V(S_2)\times V(H)\}$ where $1\leq j\leq n-1 $. When $n=2\ and \ m=2$ the product of the total graph of the path is isomorphic to  $(K_{3})\times (K_{3})$. Therefore we consider the product of the total graph of path, when $n,m > 2$.  The proof is by fixing s first and applying induction on $t$. For $s,t=1$
we need to prove the  Graham Conjecture $\pi((T(p_{n}))\times (T(P_{m}))\leq \pi(T(P_{n}))\pi(T(P_{m}))= ((2^{n-1}+(n-2))) (2^(m-1)+(m-2))$. 

Let $D$ be a distribution of $((2^{n-1}+(n-2))) (2^(m-1)+(m-2))$ pebbles on  $(T(P_n))\times (T(P_m))$.

\textbf{Case 1:} Let $\beta= (x_1, v)$, where $v\in H$. .\\
If $p_1^{1}\geq \pi(H)$, then we can transfer $1$ pebble to $(x_1, v_1)$. So let $p_1^{1}< \pi(H)$. If $\frac{p_{2}^{1}+q_{2}^{1}}{2} > \pi(H)$ we can transfer $2$ pebbles to $(x_2, v_1)$. And $(x_2, v_1)\underrightarrow{1}\  (x_1, v_1)$.  If $\frac{p_{1}^{2}+q_{1}^{2}}{2} > \pi(H)$, we can transfer $2$ pebbles to $(y_1, v_1).$  And then $(y_1, v_1) \underrightarrow{1}\  (x_1, v_1)$. If  $\frac{p_{2}^{2}+q_{2}^{2}}{2} > 2\pi(H)$ we can transfer $4$ pebbles to $(y_2, v).$ And then $(y_2, v_1) \underrightarrow{2}\  (y_1, v_1) \underrightarrow{1}\  (x_1, v_1)$. If $i\in \{1,2,\cdots \ n\}$ such that $\frac{p_{i}^{1}+q_{i}^{1}}{2} > 2^{n-2}\pi(H)$, then we can transfer $2$ pebble to $(x_2, v_1)$  and $1$ pebble to $(x_1,v_1)$. If there is $j\in \{1, 2,\cdots n-1\}$, such that    $\frac{p_{j}^{2}+q_{j}^{2}}{2} > 2^{n-2}\pi(H)$, then we can transfer $2$ pebbles to $(y_1, v_1)$  and $1$ pebble to $(x_1,v_1)$.  If  $\frac{p_{i}^{1}+q_{i}^{1}}{2}+\frac{p_{j}^{2}+q_{j}^{2}}{2} > 2^{n-2}\pi(H)$, then we can transfer $2$ pebbles to $(x_2, v_1)$ and $1$ pebble to $(x_1,v_1)$. \\
Let  $\frac{p_{i}^{1}+q_{i}^{1}}{2} < 2^{n-2}\pi(H)$ \ or \ $\frac{p_{j}^{2}+q_{j}^{2}}{2} < 2^{n-2}\pi(H)$. Since we can transfer $2$ pebbles to  $(x_2, v_1)$ or $(y_1,v_1)$ which is the neighbourhood vertex of $(x_1,v_1)$, we have one of the following equations:
\begin{equation}
	p_1^{1} + \sum_{i=2}^{n} \frac{p_{i}^{1}-q_{i}^{1}}{2} \geq \pi(H)
\end{equation}

\begin{equation}
	p_1^{1} + \sum_{j=1}^{n-1} \frac{p_{j}^{2}-q_{j}^{2}}{2} \geq \pi(H)
\end{equation}

\begin{equation}
	p_1^{1} + \sum_{i=2}^{n} \frac{p_{i}^{1}-q_{i}^{1}}{2} +\sum_{j=1}^{n-1} \frac{p_{i}^{2}-q_{i}^{2}}{2}\geq \pi(H)
\end{equation}

Thus, we can transfer $1$ pebble to the destination vertex.   Similarly, we can prove for the vertices of  $\{\{y_1\times H\},\{ y_{n-1}\times H\}, \{\{x_n\}\times H\}$.

\textbf{Case 2:} Let $\beta= (x_a, v)$, where $v\in H$ and $2\leq a \leq n-1$.\\
If $p_a^{1}\geq \pi(H)$, then we can transfer $1$ pebble to $(x_a, v)$. So let $p_a^{1}< \pi(H)$. If   $\frac{p_{a-1}^{1}+q_{a-1}^{1}}{2} > \pi(H)$ then we can transfer $2$	
pebbles to $(x_{a-1}, v)$. And then $(x_{a-1}, v) \underrightarrow{1}\  (x_a,v)$. If $\frac{p_{a}^{2}+q_{a}^{2}}{2} > \pi(H)$ then we can transfer $2$	 
pebbles to $(y_{a}, v)$. And then $(y_{a}, v)\underrightarrow{1}\  (x_a,v)$. If  $\frac{p_{a-1}^{2}+q_{a-1}^{2}}{2} > 2\pi(H)$  then we can transfer $4$ pebbles to  $(y_{(a-1)}, v)$. And then $(y_{(a-1)}, v)\ \underrightarrow{2}\ (x_(a-1),v)\ \underrightarrow{1}\ (x_a, v)$. If   $i\in \{1,2,3\cdots \ n\}, \ i\ne a \ and \ i > a $  such that   $\frac{p_{i}^{1}+q_{i}^{1}}{2} > 2^{n-2-a}\pi(H)$, then we can transfer $2$ pebbles to $(x(a+1),v)$  and $1$ pebble to $(x_a,v)$. If  $j\in \{1,2,3\cdots \ n-1\}, \ j\ne a\ and \ j > a $  such that  $\frac{p_{j}^{2}+q_{j}^{2}}{2} > 2^{n-2-a}\pi(H)$, then we can transfer $2$ pebbles to $(y_a, v)$  and $1$ pebble to $(x_a,v)$. If   $\frac{p_{i}^{1}+q_{i}^{1}}{2}+\frac{p_{j}^{2}+q_{j}^{2}}{2} > 2^{n-2-a}\pi(H)$, then we can transfer $2$ pebbles to either $(x_{a+1},v)$\ or \ $(y_{a}, v)$ and $1$ pebble to $(x_a,v)$.  If  $i\in \{1,2,3\cdots \ n\}, \ i\ne a \ and \ i < a $  such that   $\frac{p_{i}^{1}+q_{i}^{1}}{2} > 2^{a-2}\pi(H)$, then we can transfer $2$ pebbles to $(x_{a-1},v)$  and $1$ pebble to $(x_a,v)$. If  $j\in \{1,2,3\cdots \ n-1\}, \ j\ne a\ and \ j < a $  such that  $\frac{p_{j}^{2}+q_{j}^{2}}{2} > 2^{a-2}\pi(H)$, then we can transfer $2$ pebbles to $(y_a, v)$  and $t$ pebbles to $(x_a,v)$. If $i,j<a$, such that  $\frac{p_{i}^{1}+q_{i}^{1}}{2}+\frac{p_{j}^{2}+q_{j}^{2}}{2} > 2^{a-2}\pi(H)$, then we can transfer $2$ pebbles to either $(x_{a-1},v)$\ or \ $(y_{a}, v)$ and $1$ pebble to $(x_a,v)$\\
Let  $\frac{p_{i}^{1}+q_{i}^{1}}{2} < 2^{n-2-a}\pi(H) \ for\ i>a$ \ or \ $\frac{p_{j}^{2}+q_{j}^{2}}{2} < 2^{n-2-a}\pi(H)\ for\ j>a$ \ or \
$\frac{p_{i}^{1}+q_{i}^{1}}{2} < 2^{a-2}\pi(H) \ for\ i<a$ \ or \ $\frac{p_{j}^{2}+q_{j}^{2}}{2} < 2^{a-2}\pi(H)\ for\ j>a$

Since we can transfer $2$ pebbles to any vertex of the neighbourhood  of $(v_a,x)$, we have one of the following equations:
\begin{equation}
	p_a^{1} + \sum_{i=1,i\ne a}^{n} \frac{p_{i}^{1}-q_{i}^{1}}{2} \geq \pi(H)
\end{equation}

\begin{equation}
	p_a^{2} + \sum_{j=1,\ j\ne a}^{n-1} \frac{p_{j}^{2}-q_{j}^{2}}{2} + \geq \pi(H)
\end{equation}

\begin{equation}
	p_a^{1} + \sum_{i=1, i\ne a, \ i> a \ or \ i< a}^{n} \frac{p_{i}^{1}-q_{i}^{1}}{2} +\sum_{j=1, j\ne a, \ j< a,\ or \ j>a}^{n-1} \frac{p_{j}^{2}-q_{j}^{2}}{2}\geq \pi(H)
\end{equation}

Thus, we can transfer $1$ pebble to the destination.   Similarly, we can prove for the vertices in  $\{\{V(S_1)\}\times H- \{(x_a,v), (x_n, v)\}\},\{\{V(S_2)\}\times H- \{(y_1,v), (y_{n-1}, v)\}\}$.\\

When $s=1\ and \ t\geq 2$ it follows from Lourdusamy Conjecture. We assume that the theorem is true for $2\leq s^{'}, t^{'}< s,t.$  First let us try to shift $t$ pebbles to the desired vertex by induction on $t$ using the Lourdusamy pebbling conjecture  and then shift $(s-1)t$ pebbles to the desired vertex using induction on $s$.\\

Let $D$ be a distribution of $(s(2^{n-1}+(n-2))) (t2^(m-1)+(m-2))$ pebbles on the vertices of $(T(P_n))\times (T(P_m))$.

\textbf{Case 3:} Let $\beta=(x_1, v)$, where $v\in H$.\\
If $p_1^{1}\geq \pi_t(H)$, then we can transfer $t$ pebbles to $(x_1, v)$. So let $p_1^{1}< \pi_t(H)$. If $\frac{p_{2}^{1}+q_{2}^{1}}{2} > \pi_t(H)$ we can move $2t$ pebbles to $(x_2, v)$. And then $(x_2, v)\underrightarrow{t}\  (x_1, v)$.  If  $\frac{p_{1}^{2}+q_{1}^{2}}{2} > \pi_t(H)$ we can move $2t$ pebbles to $(y_1, v)$. And then $(y_1, v)\underrightarrow{t}\  (x_1, v)$. If   $\frac{p_{2}^{2}+q_{2}^{2}}{2} > 2\pi_t(H)$ we can move $4t$ pebbles to $(y_2, v)$. And then $(y_2, v) \underrightarrow{2t}\  (y_1, v) \underrightarrow{t}\  (x_1, v)$. If $i\in \{1,2,\cdots \ n\}$ such that $\frac{p_{i}^{1}+q_{i}^{1}}{2} > 2^{n-2}\pi_t(H)$, then we can transfer $2t$ pebbles to $(x_2, v)$  and $t$ pebbles to $(x_1,v)$. If $j\in \{1, 2,\cdots n-1\}$, such that    $\frac{p_{j}^{2}+q_{j}^{2}}{2} > 2^{n-2}\pi_t(H)$, then we can transfer $2t$ pebbles to $(y_1, v)$  and $t$ pebbles to $(x_1,v)$.  If  $\frac{p_{i}^{1}+q_{i}^{1}}{2}+\frac{p_{j}^{2}+q_{j}^{2}}{2} > 2^{n-2}\pi_t(H)$, then we can transfer $2t$ pebbles to $(x_2, v)$ and $t$ pebble to $(x_1,v)$. \\
Let  $\frac{p_{i}^{1}+q_{i}^{1}}{2} < 2^{n-2}\pi_t(H)$ \ or \ $\frac{p_{j}^{2}+q_{j}^{2}}{2} < 2^{n-2}\pi_t(H)$
Since we can transfer $2t$ pebbles to  $(x_2, v)$ or $(y_1,v)$ which is the  neighbourhood vertex of $(x_1,v)$, we have one of the following equations:
\begin{equation}
	p_1^{1} + \sum_{i=2}^{n} \frac{p_{i}^{1}-q_{i}^{1}}{2} \geq \pi_t(H)
\end{equation}

\begin{equation}
	p_1^{1} + \sum_{j=1}^{n-1} \frac{p_{j}^{2}-q_{j}^{2}}{2} \geq \pi_t(H)
\end{equation}

\begin{equation}
	p_1^{1} + \sum_{i=2}^{n} \frac{p_{i}^{1}-q_{i}^{1}}{2} +\sum_{j=1}^{n-1} \frac{p_{i}^{2}-q_{i}^{2}}{2}\geq \pi_t(H)
\end{equation}

Thus, we can transfer $t$ pebble to $\beta$.  The count  of pebbles remaining on the graph is $(s2^{n-1)}(n-2))\pi_t(H)- \pi_t(H) \geq ((s-1)2^{n-1}+(n-2))\pi_t(H)=\pi_{s-1}(G)\pi_t(H).$ Thus, we can transfer $(s-1)t$ pebbles additionally  to $\beta$. Thus, $\beta$ gets $st$ pebbles and we are done. Similarly, we can prove for $\{\{y_1, y_{n-1}\}\times H\}, \{\{x_n\}\times H\}$. 



\textbf{Case 4:} Let $\beta=(x_a, v)$, where $v\in H$ and $2\leq a \leq n-1$.\\
If $p_a^{1}\geq \pi_t(H)$, then we can transfer $t$ pebbles to $(x_a, v)$. So let $p_a^{1}< \pi_t(H)$. If  $\frac{p_{a-1}^{1}+q_{a-1}^{1}}{2} > \pi_t(H)$ then we can transfer $2t$	
pebbles to $(x_{a-1}, v)$. And then $(x_{a-1}, v)\underrightarrow{t}\  (x_a,v)$. If   $\frac{p_{a}^{2}+q_{a}^{2}}{2} > \pi_t(H)$ then we can transfer $2t$	 
pebbles to $(y_{a}, v)$. And then $(y_{a}, v)\underrightarrow{t}\  (x_a,v)$. If  $\frac{p_{a-1}^{2}+q_{a-1}^{2}}{2} > 2\pi_t(H)$  then we can transfer $4t$ pebbles to  $(y_{(a-1)}, v)$. And then $(y_{(a-1)}, v)\ \underrightarrow{2t}\ (x_{a-1},v)\ \underrightarrow{t}\ (x_a, v)$. If   $i\in \{1,2,3\cdots \ n\}, \ i\ne a \ and \ i > a $  such that   $\frac{p_{i}^{1}+q_{i}^{1}}{2} > 2^{n-2-a}\pi_t(H)$, then we can transfer $2t$ pebbles to $(x(a+1),v)$  and $t$ pebbles to $(x_a,v)$. If   $j\in \{1,2,3\cdots \ n-1\}, \ j\ne a\ and \ j > a $  such that  $\frac{p_{j}^{2}+q_{j}^{2}}{2} > 2^{n-2-a}\pi_t(H)$, then we can transfer $2t$ pebbles to $(y_a, v)$  and $t$ pebbles to $(x_a,v)$. If  $\frac{p_{i}^{1}+q_{i}^{1}}{2}+\frac{p_{j}^{2}+q_{j}^{2}}{2} > 2^{n-2-a}\pi_t(H)$, then we can transfer $2t$ pebbles to either $(x_{a+1},v)$\ or \ $(y_{a}, v)$ and $t$ pebbles to $(x_a,v)$.  If  $i\in \{1,2,3\cdots \ n\}, \ i\ne a \ and \ i < a $  such that   $\frac{p_{i}^{1}+q_{i}^{1}}{2} > 2^{a-2}\pi_t(H)$, then we can transfer $2t$ pebbles to $(x(a-1),v)$  and $t$ pebbles to $(x_a,v)$. If  $j\in \{1,2,3\cdots \ n-1\}, \ j\ne a\ and \ j < a $  such that  $\frac{p_{j}^{2}+q_{j}^{2}}{2} > 2^{a-2}\pi_t(H)$, then we can transfer $2t$ pebbles to $(y_a, v)$  and $t$ pebbles to $(x_a,v)$. If there are  $i,j<a$, such that  $\frac{p_{i}^{1}+q_{i}^{1}}{2}+\frac{p_{j}^{2}+q_{j}^{2}}{2} > 2^{a-2}\pi_t(H)$, then we can transfer $2t$ pebbles to either $(x_{a-1},v)$\ or \ $(y_{a}, v)$ and $t$ pebbles to $(x_a,v)$\\
Let  $\frac{p_{i}^{1}+q_{i}^{1}}{2} < 2^{n-2-a}\pi_t(H) \ \forall i>a$ \ or \ $\frac{p_{j}^{2}+q_{j}^{2}}{2} < 2^{n-2-a}\pi_t(H)\ \forall j>a$ \ or \
$\frac{p_{i}^{1}+q_{i}^{1}}{2} < 2^{a-2}\pi_t(H) \ for i<a$ \ or \ $\frac{p_{j}^{2}+q_{j}^{2}}{2} < 2^{a-2}\pi_t(H)\ for j>a$

Since we can transfer $2t$ pebbles to any vertex of the neighbourhood  of $(v_a,x)$, we have one of the following equations:
\begin{equation}
	p_a^{1} + \sum_{i=1,i\ne a}^{n} \frac{p_{i}^{1}-q_{i}^{1}}{2} \geq \pi_t(H)
\end{equation}

\begin{equation}
	p_a^{2} + \sum_{j=1,\ j\ne a}^{n-1} \frac{p_{j}^{2}-q_{j}^{2}}{2} + \geq \pi_t(H)
\end{equation}

\begin{equation}
	p_a^{1} + \sum_{i=1, i\ne a, \ i> a \ or \ i< a}^{n} \frac{p_{i}^{1}-q_{i}^{1}}{2} +\sum_{j=1, j\ne a, \ j< a,\ or \ j>a}^{n-1} \frac{p_{j}^{2}-q_{j}^{2}}{2}\geq \pi_t(H)
\end{equation}

Thus, we can transfer $t$ pebble to $\beta$.  The number pebbles remaining on the graph is $(s2^{n-1}+(n-2))\pi_t(H)- \pi_t(H) \geq ((s-1)2^{n-1}+(n-2))\pi_t(H)=\pi_{s-1}(G)\pi_t(H).$ Thus we can move $(s-1)t$ pebbles additionaly  to $\beta$. Thus, $\beta$ gets $st$ pebbles and we are done. Similarly, we can prove for the vertices in $\{\{V(S_1)\}\times H- \{(x_a,v), (x_n, v)\}\},\{\{V(S_2)\}\times H- \{(y_1,v), (y_{n-1}, v)\}\}$.\\

\end{proof}

\section{Conclusion}
In the present article, we compute pebbling number and $t$-pebbling number and $2t-$ pebbling property of the total graph of the path. Then we prove the Herscovici's Conjecture $\pi_{st}(T(P_n))\times (T(P_m))\leq \pi_s(T(P_n))\pi_t(T(P_m))= (s(2^{n-1}+(n-2))) (t2^(m-1)+(m-2))$.

\end{document}